\documentclass[12pt]{article}
\usepackage{amsthm,amsfonts,amssymb,amsmath}
\usepackage{enumerate}
\usepackage[colorlinks=true]{hyperref}

\newcommand{\BC}{{\mathbb {C}}}

\newcommand{\BK}{{\mathbb {K}}}

\newcommand{\BN}{{\mathbb {N}}}

\newcommand{\BP}{{\mathbb {P}}}
\newcommand{\BQ}{{\mathbb {Q}}}
\newcommand{\BR}{{\mathbb {R}}}

\newcommand{\CA}{{\mathcal {A}}}
\newcommand{\CB}{{\mathcal {B}}}
\newcommand{\CD}{{\mathcal {D}}}
\newcommand{\CE}{{\mathcal {E}}}

\newcommand{\CH}{{\mathcal {H}}}
\newcommand{\CI}{{\mathcal {I}}}

\newcommand{\CL}{{\mathcal {L}}}
\newcommand{\CM}{{\mathcal {M}}}

\newcommand{\CO}{{\mathcal {O}}}

\newcommand{\CQ}{{\mathcal {Q}}}

\newcommand{\CU}{{\mathcal {U}}}
\newcommand{\CV}{{\mathcal {V}}}
\newcommand{\CW}{{\mathcal {W}}}
\newcommand{\CX}{{\mathcal {X}}}

\newcommand{\an}{{\mathrm{an}}}
\newcommand{\ran}{{\text{r-an}}}

\newcommand{\Div}{{\mathrm{Div}}}

\renewcommand{\div}{{\mathrm{div}}}

\newcommand{\Hom}{{\mathrm{Hom}}}

\newcommand{\intb}{{\mathrm{int}}}

\newcommand{\NS}{{\mathrm{NS}}}

\newcommand{\Pic}{\mathrm{Pic}}

\newcommand{\CaCl}{\mathrm{CaCl}}

\newcommand{\Picc}{{\mathcal{P}\mathrm{ic}}}

\newcommand{\Prep}{\mathrm{Prep}}

\newcommand{\rmod}{{\mathrm{mod}}}

\DeclareMathOperator{\Spec}{Spec}

\newcommand{\wt}{\widetilde}
\newcommand{\wh}{\widehat}

\newcommand{\lra}{\longrightarrow}

\newcommand{\fh}{\mathfrak{h}}

\newcommand{\nef}{\mathrm{nef}}
\renewcommand{\vert}{\mathrm{vert}}

\newcommand{\Ber}{\mathrm{Ber}}

\newcommand{\CLL}{\overline{\mathcal L}}
\newcommand{\CMM}{\overline{\mathcal M}}
\newcommand{\CHH}{\overline{\mathcal H}}

\newcommand{\CAA}{\overline{\mathcal A}}
\newcommand{\CDD}{\overline{\mathcal D}}
\newcommand{\CEE}{\overline{\mathcal E}}

\newcommand{\CC}{\mathbb{C}}
\newcommand{\RR}{\mathbb{R}}
\newcommand{\ZZ}{\mathbb{Z}}
\newcommand{\QQ}{\mathbb{Q}}

\newcommand{\DS}{\mathcal{DS}}

\newtheorem{thm}{Theorem}[section]

\newtheorem{lem}[thm]{Lemma}

\newtheorem{defn}[thm]{Definition}

\theoremstyle{definition}

\theoremstyle{remark}

\begin{document}

\title{The arithmetic Hodge index theorem for adelic line bundles II:
finitely generated fields}
\author{Xinyi Yuan, Shou-Wu Zhang}
\maketitle

\tableofcontents

\section{Introduction}

This paper is a sequel to \cite{YZ1}, which lies in the fields of Arakelov geometry and algebraic dynamics.

In \cite{YZ1}, we have proved a Hodge index theorem for adelic line bundles over number fields, and applied it to prove a rigidity theorem of preperiodic points of polarized algebraic dynamical systems over number fields. 
The treatment is based on the theory of adelic line bundles over projective varieties over number fields introduced in \cite{Zh2}. 
In \cite{YZ2}, we have generalized the theory of \cite{Zh2} to a theory of adelic line bundles over quasi-projective varieties over finitely generated fields. 
The goal of the current paper, based on the theory in \cite{YZ2}, 
 is to extend the results of \cite{YZ1} to projective varieties over finitely generated fields over $\QQ$. 
More precisely, we prove both the Hodge index theorem for adelic line bundles and the rigidity theorem of preperiodic points over finitely generated fields over $\QQ$. 

As the treatment of \cite{YZ1} over number fields is generalized to function fields of one variable by Carney \cite{Ca1}, the treatment of the current paper is generalized to
finitely generated function fields by Carney \cite{Ca2}.

Throughout this paper, denote $M_\QQ=M\otimes_\ZZ\QQ$ for any abelian group $M$.
Denote $\CX_\QQ=\CX\times_{\Spec \ZZ}\Spec\QQ$ for any scheme $\CX$ over $\ZZ$.

\subsection{Arithmetic Hodge index theorem}

In \cite{YZ1}, we have proved an arithmetic Hodge index theorem for adelic line bundles over a projective variety over a number field, which extends the previous results of Faltings \cite{Fal}, Hriljac \cite{Hr} and Moriwaki \cite{Mo2}.
Here we describe the generalization to finitely generated fields. 

We need the theory of adelic $\QQ$-line bundles over finitely generated fields in \cite{YZ2}.  We refer to \S\ref{sec adelic} for a quick review of the theory. 
Namely, let $X$ be a projective variety over a finitely generated field $K$ over $\QQ$. Then we have the group $\wh \Pic (X)$ of isomorphism classes of 
\emph{adelic $\QQ$-line bundles} over $X$. Note that the notation $\wh \Pic (X)$ in this paper is equivalent to the notation 
$\wh \Pic (X/\ZZ)_\QQ$ in \cite{YZ2}. 

Moreover, we have the sub-semigroup $\wh \Pic (X)_\nef$ of 
\emph{nef} adelic $\QQ$-line bundles in $\wh \Pic (X)$, and the subgroup $\wh \Pic (X)_\intb$ of
\emph{integrable} adelic $\QQ$-line bundles in $\wh \Pic (X)$. 

There are two intersection pairings involved in our treatment. 
The first one is an absolute intersection product 
$$
\wh \Pic(K)_\intb ^{d+1}\lra \RR,
$$
and the second one is a Deligne pairing
$$
\pi_*:\wh \Pic(X)_\intb ^{n+1}\lra \wh \Pic(K)_\intb.
$$
Here assume that $K$ has transcendental degree $d$ over $\QQ$, and
$X$ has dimension $n$, and $\pi:X\to \Spec K$ denotes the structure morphism.

To state our Hodge index theorem, we introduce the following further positivity notions. 

\begin{defn} \label{def-positivity} 
Let $K$ be a finitely generated field over $\QQ$ of transcendental degree $d$.
Let $X$ be a projective variety over $K$.
Let $\overline L, \overline M\in \wh \Pic (X)_\intb$
and $\overline H\in \wh \Pic(K)_\intb$. We define the following notions. 
\begin{enumerate} [(1)]
\item $\overline H\succeq 0$ if $\overline H$ is \emph{pseudo-effective}, i.e., the top intersection number $\overline H\cdot \overline N_1\cdots \overline N_d\geq 0$ for any $\overline N_1, \cdots, \overline N_d$ in $\wh \Pic (K)_\nef$. 

\item $\overline H \equiv 0$ if $\overline H$ is \emph{numerically trivial}, i.e., the top intersection number $\overline H\cdot \overline N_1\cdots \overline N_d= 0$ for any $\overline N_1, \cdots, \overline N_d$ in $\wh \Pic (K)_\intb$. 

\item $\overline L\gg 0$ if  $L$ is ample, and $\overline L- \overline N$ is nef for some $\overline N\in \wh\Pic(\QQ)$ with $\wh\deg (\overline N)>0$. 
The adelic line bundle $\overline N$ is viewed as an element of $\wh \Pic (X)_\intb$ by the natural pull-back map. 

\item $\overline M$ is \emph{$\overline L$-bounded} if  there is a rational number $\epsilon >0$ such that both 
$\overline L+ \epsilon \overline M$ and $\overline L- \epsilon \overline M$ are nef. 
\end{enumerate}
\end{defn}

\medskip

It is conventional that each of $\overline H_1\preceq \overline H_2$ and $\overline H_2\succeq \overline H_1$ means $\overline H_2-\overline H_1\succeq 0$. Similar conventions apply to 
``$\equiv$'' and ``$\gg$.''
The main theorem of this paper is as follows.

\begin{thm}\label{hodge} 
Let $K$ be a finitely generated field over $\QQ$, and $\pi:X\to \Spec K$ be a normal, geometrically connected, and projective variety of dimension $n\geq 1$. 
Let $\overline M$ be an integrable adelic $\QQ$-line bundle on $X$, and $\overline L_1, \cdots, \overline L_{n-1}$ be $n-1$ nef adelic $\QQ$-line bundles on $X$ where each $L_i$ is big on $X$. 
Assume $M\cdot L_1\cdots L_{n-1}=0$ on $X$. 
Then 
$$\pi_*(\overline M^2\cdot \overline L_1\cdots \overline L_{n-1}) \preceq 0$$
in $\wh \Pic(K)_\intb$.

Moreover, if $\overline L_i\gg 0$, and $\overline M$ is $\overline L_i$-bounded for each $i$, then
$$\pi_*(\overline M^2\cdot \overline L_1\cdots \overline L_{n-1}) \equiv 0$$
in $\wh \Pic(K)_\intb$
if and only if $\overline M\in \pi^* \wh \Pic(K)_\intb$.
\end{thm} 

When $K$ is a number field, the theorem was proved in \cite[Theorem 3.2]{YZ1}, the pre-sequel of this theorem. The proof of this theorem actually follows the line of that in \cite{YZ1}, replacing adelic line bundles of \cite{Zh2} by those of \cite{YZ2}.

The theorem is generated to finitely generated fields $K$ over a base field $k$ by Carney \cite{Ca1, Ca2}.

\subsection{Preperiodic points of algebraic dynamics}

Let $K$ be a field. 
A \emph{polarized algebraic dynamical system over $K$} is a triple $(X,f,L)$, where
$X$ is a projective variety over $K$,
$f:X\to X$ is a $K$-morphism, and $L\in \Pic(X)_\BQ$ is an ample $\QQ$-line bundle  
 satisfying $f^*L =qL$ from some rational number $q>1$. We call $L$ a \emph{polarization of $f$}.
Denote by $\Prep (f)$ the set of \emph{preperiodic
points}, i.e.,
$$
\Prep(f):=\{x\in X(\overline K)\ |\ f^m(x)=f^n(x) {\rm\ for\ some\ } m,n\in \BN,
\ m\neq n \}.
$$
A well-known result of Fakhruddin \cite{Fak} asserts that $\Prep (f)$ is
Zariski dense in $X$.

Fix a projective variety $X$ over $K$.
Denote by $\DS(X)$ the set of all morphism $f:X\to X$ over $K$ that is \emph{polarizable}; i.e., there is an ample $\QQ$-line bundle $L\in \Pic(X)_\BQ$ such that $(X,f,L)$ forms a polarized dynamical system. 
Note that we do \emph{not} require elements of $\DS(X)$ to be polarizable by the same ample line bundle. 

\begin{thm}\label{dynamicsmain}
Let $X$ be a projective variety over a field $K$ of characteristic 0.
For any $f,g\in \DS(X)$, the following are equivalent:
\begin{itemize}
\item[(1)] $\Prep(f)= \Prep(g)$;
\item[(2)] $g\Prep(f)\subset\Prep(f)$;
\item[(3)] $\Prep(f)\cap \Prep(g)$ is Zariski dense in $X$.
\end{itemize}
\end{thm}

When $X=\BP^1$, the theorem is independently proved by M. Baker and L. DeMarco
\cite{BD} during the preparation of this paper. 
Their treatment for the number field case is the same as our treatment in the earlier version, while the method for the general case is quite
different.

Similar to Theorem \ref{hodge}, the theorem is treated in \cite{YZ1} if $K$ is a number field, and generalized to any field of positive characteristic by Carney \cite{Ca1, Ca2}.

Some consequences and questions related to the theorem can be found in 
\cite[\S 4.5]{YZ1}. Here we only list the following one, whose proof will be given at the end of this paper. 

\begin{thm}\label{dynamicslocal}
Let $\mathbb K$ be either $\BC$ or $\BC_p$ for some prime number $p$.
Let $X$ be a projective variety over $\mathbb K$, and $f,g\in \DS(X)$ be two polarizable algebraic dynamical systems.
If $\Prep(f)\cap\Prep(g)$ is Zariski dense in $X$, then $d\mu_f=d\mu_g$.
\end{thm}

Here $d\mu_f$ denotes the equilibrium measure of $(X, f)$ on the Berkovich space $X^\Ber$ associated to $X$. It can
be obtained from any initial ``smooth'' measure on $X^\Ber$ by Tate's limiting
argument. By a proper interpretation, it satisfies $f^*d\mu_f=q^{\dim X}d\mu_f$ and $f_*d\mu_f=d\mu_f$.

The proof of Theorem \ref{dynamicsmain} follows the idea in \cite{YZ1}, obtained as a consequence of Theorem \ref{hodge}. 
The Lefschetz principle allows us to assume that the base field $K$ is a finitely generated field over $\QQ$. 
Then the problem sits in the framework of adelic line bundles and height theory over finitely generated fields in \cite{YZ2}.

As in an earlier draft of our paper, one may prove Theorem \ref{dynamicsmain}
without using the new theory of adelic line bundles over finitely generated fields in \cite{YZ2}. 
The proof combines the height theory of \cite{Mo3,Mo4}, the equidistribution idea of \cite{SUZ} and \cite{Yu}, and the results of \cite{YZ1}. But the proof is very tricky and very technical due to many convergence problems of integrations and intersection numbers.

\subsubsection*{Acknowlegement}

The authors would like to thank Xander Faber, Dragos Ghioca, Walter Gubler,
Yunping Jiang,
Barry Mazur, Thomas Tucker, Yuefei Wang, Chenyang Xu, Shing-Tung Yau, Yuan Yuan, Zhiwei Yun
and Wei Zhang for many helpful discussions during the preparation of
this paper.

During the preparation of the paper, the first author was supported by the Clay Research Fellowship and the NSF award DMS-1330987 of the USA, and the
second author was supported by the NSF awards DMS-0970100, DMS-1065839,  and DMS-1700883 of the USA.

\section{Arithmetic Hodge index theorem}

The goal of this section is to prove Theorem \ref{hodge}. 
We will start with a review of the notion of 
adelic $\QQ$-line bundles in \cite{YZ2}, and then prove the theorem following the idea of \cite{YZ1}.

\subsection{Review on adelic line bundles} \label{sec adelic}

Let us first recall the notion of adelic divisors and adelic line bundles over finitely generated fields over $\QQ$ in \cite[\S2.4,\S2.5]{YZ2}. We will be concerned with only adelic $\QQ$-line bundles (instead of integral adelic line bundles), which is slightly simpler without treating mixed coefficients. 
We only need the arithmetic case ($k=\ZZ$), which also avoids the uniform terminology of \cite[\S1.6]{YZ2}.

\subsubsection*{Adelic divisors over a quasi-projective arithmetic variety}

Let us review the notion of adelic $\QQ$-divisors, as a slight variant of that in  \cite[\S2.4]{YZ2}.

By a \emph{quasi-projective arithmetic variety} (resp. \emph{projective arithmetic variety}), we mean an integral scheme which is flat and quasi-projective (resp. projective) over $\ZZ$. 
For a quasi-projective arithmetic variety $\CU$, a \emph{projective model} means a projective arithmetic variety $\CX$ endowed with an open immersion $\CU\to \CX$. 

Let $\CU$ be a quasi-projective arithmetic variety, and $\CX$ be a projective model of $\CU$. 
Denote by $\wh\Div(\CX)$ the group of arithmetic divisors over $\CX$, and by $\wh\Pr(\CX)$ the subgroup of principal arithmetic divisors over $\CX$.
Here the Green's functions are assumed to be \emph{continuous} as in 
\cite[\S2.1]{YZ2}.
Projective models $\CX$ of $\CU$ form an inverse system. 
Using pull-back maps, define 
$$\wh\Div (\CU)_{\rmod,\QQ}:=\lim_{\substack{\lra\\ \CX}}\wh\Div(\CX)_\QQ,
\qquad \wh\Pr (\CU)_{\rmod,\QQ}:=\lim_{\substack{\lra\\ \CX}}\wh\Pr(\CX)_\QQ.$$

An element of $\wh\Div (\CU)_{\rmod,\QQ}$ is called \emph{effective} if it is the image of an effective arithmetic divisor of some $\wh\Div(\CX)_\QQ$.
For any $\overline \CD, \overline \CE \in \wh\Div (\CU)_{\rmod,\QQ}$, we write 
$\overline \CD\geq \overline \CE$ or $\overline \CE\leq \overline \CD$ if $\overline \CD-\overline \CE$ is effective in $\wh\Div (\CU)_{\rmod,\QQ}$.
This give a partial order in $\wh\Div (\CU)_{\rmod,\QQ}$.

Fix a \emph{boundary divisor} $(\CX_0,\CDD_0)$ of $\CU$, i.e,  a projective model $\CX_0$ of $\CU$ and a strictly effective arithmetic divisor $\CDD_0$ over $\CX_0$ such that the support of the finite part $\CD_0$ is exactly $\CX_0\setminus \CU$.
We have a \emph{boundary norm} 
$$\|\cdot\|_{\CDD_0}:\wh\Div (\CU)_{\rmod,\QQ}
\lra [0,\infty]$$
by 
$$
\|\CEE\|_{\CDD_0}:=\inf\{\epsilon\in \BQ_{>0}: \ 
 -\epsilon \CDD_0 \leq
\CEE \leq  \epsilon \CDD_0\}.
$$
Here the inequalities are defined in terms of effectivity. 
It further induces a \emph{boundary topology} over $\wh\Div (\CU)_{\rmod}$, which does not depend on the choice of $(\CX_0,\overline\CD_0)$.

Let $\wh \Div  (\CU)_{\QQ}$ be the \emph{completion} of $\wh \Div  (\CU)_{\rmod,\QQ}$ with respect to the boundary topology. 
By definition, an element of $\wh \Div  (\CU)_{\QQ}$ is represented by a Cauchy sequence in $\wh \Div  (\CU)_{\rmod,\QQ}$, i.e., a sequence $\{\CEE_i\}_{i\geq 1}$ in $\wh \Div  (\CU)_{\rmod,\QQ}$ satisfying the property that there is a sequence $\{\epsilon_i\}_{i\geq 1}$ of positive rational numbers converging to $0$ such that 
$$
 -\epsilon_i \overline\CD_0 \leq
\overline \CE_i-\CEE_{i'} \leq  \epsilon_i \overline \CD_0,\quad\ i'\geq i\geq 1.
$$
The sequence $\{\CEE_i\}_{i\geq 1}$ represents $0$ in $\wh \Div  (\CU)_{\QQ}$ if and only if there is a sequence $\{\delta_i\}_{i\geq 1}$ of positive rational numbers converging to $0$ such that 
$$
 -\delta_i \overline\CD_0 \leq
 \CEE_i \leq  \delta_i \overline \CD_0,\quad\ i\geq 1.
$$

An element of $\wh \Div  (\CU)_{\QQ}$ is called an \emph{adelic $\QQ$-divisor} over $\CU$.
Define the \emph{class group} of $\CU$ to be 
$$\wh\CaCl (\CU)_{\QQ}:=\wh\Div(\CU)_{\QQ}/\wh \Pr (\CU)_{\rmod,\QQ}.$$

\subsubsection*{Category of $\QQ$-line bundles}

As a convention, categories of various line bundles (or $\QQ$-line bundles) are defined to be groupoids; i.e., the morphisms in them are defined to be isomorphisms (or isometries) of the line bundles. 

To illustrate the idea, take $\wh\Picc(\CX)$
and $\wh\Picc(\CX)_\QQ$ for a projective arithmetic variety $\CX$ for example. 
Then an object of $\wh\Picc(\CX)$ is a hermitian line bundle over $\CX$ (with a continuous metric), and a morphism of two such objects is defined to be
an isometry of hermitian line bundles. 
An object of $\wh\Picc(\CX)_\QQ$ is a pair $(a,\CL)$ (or just written as $a\CL$)
with $a\in \QQ$ and $\CL\in\wh\Picc(\CX)$, and a morphism of two such objects is defined to be
$$\Hom(a\CL,a'\CL')=\varinjlim_m \Hom(am\CL, a'm\CL'),$$
where $m$ runs through positive integers such that $am$ and $a'm$ are both integers, so that $am\CL$ and $a'm\CL'$ are viewed as objects of $\wh\Picc(\CX)$, 
and ``$\Hom$'' on the right-hand side are viewed in $\wh\Picc(\CX)$. 

We refer to \cite[\S2.1,\S2.2]{YZ2} for more details on these categories.

\subsubsection*{Adelic line bundles over a quasi-projective arithmetic variety}

Now we review the notion of adelic $\QQ$-line bundles, as a slight variant of that in \cite[\S2.5]{YZ2}.

Let $\CU$ be a quasi-projective arithmetic variety.
Let $\CX$ be a projective model of $\CU$ over $k$.
In the spirit of \cite[\S2.2]{YZ2}, 
let $\wh\Picc(\CX)$ be the category of hermitian line bundles over $\CX$, 
and $\wh\Picc(\CX)_\QQ$ be the category of hermitian $\QQ$-line bundles over $\CX$. 
In the arithmetic case (that $k=\ZZ$), $\wh\Picc(\CX)$ is the category of hermitian line bundles with continuous metrics over $\CX$. 
In the geometric case (that $k$ is a field), $\wh\Picc(\CX)$ means the usual 
$\Picc(\CX)$.

Define the \emph{category of adelic $\QQ$-line bundles} $\wh\Picc (\CU)_\QQ$ over $\CU$ as follows.
An object of $\wh\Picc (\CU)_\QQ$ is a pair
$(\CL, (\CX_i,\overline \CL_i, \ell_{i})_{i\geq 1})$ where:
\begin{enumerate}[(1)]
\item $\CL$ is an object of $\Picc(\CU)_\QQ$, i.e., a $\QQ$-line bundle over $\CU$;

\item  $\CX_i$ is a projective model of $\CU$ over $\ZZ$;

\item  $\overline \CL_i$ is an object of $\wh\Picc(\CX_i)_\QQ$, i.e. a hermitian $\QQ$-line bundle over $\CX_i$;

\item $\ell_i:\CL\to \CL_i|_{\CU}$ is an isomorphism in $\Picc(\CU)_\QQ$.
\end{enumerate}
The sequence is required to satisfy the \emph{Cauchy condition} that
the sequence $\{\wh \div(\ell_i \ell_1^{-1})\}_{i\geq 1}$ is a 
Cauchy sequence in $\wh\Div(\CU)_{\rmod,\QQ}$ under the boundary topology.

A morphism from an object $(\CL, (\CX_i,\overline \CL_i, \ell_{i})_{i\geq 1})$ of $\wh\Picc (\CU)_\QQ$ to another object
$(\CL',(\CX_i',\overline \CL_i', \ell_{i}')_{i\geq 1})$ of $\wh\Picc (\CU)_\QQ$ is an isomorphism  $\iota:\CL\to \CL'$ of $\QQ$-line bundles over $\CU$ such that the sequence 
$\{\wh\div(\ell_{i}'\ell_1'^{-1})-\wh\div(\ell_{i}\ell_1^{-1})+\wh\div(\iota_1) \}_{i\geq1}$
of $\wh \Div (\CU)_{\rmod,\QQ}$ converges to 0 in $\wh \Div (\CU)_\QQ$
under the boundary topology.
Here $\iota_1:\overline \CL_1\dashrightarrow \overline \CL_1'$ denotes the rational map over $\CU$ induced by $\iota$, which induces an element $\wh\div(\iota_1)$
of $\wh \Div (\CU)_{\rmod,\QQ}$.

An object of $\wh\Picc (\CU)_{\QQ}$ is called an \emph{adelic $\QQ$-line bundle} over $\CU$.
Define $\wh\Pic (\CU)_{\QQ}$ to be the \emph{group} of isomorphism classes of objects of $\wh\Picc (\CU)_{\QQ}$. 
There is a canonical isomorphism
$$\wh\CaCl(\CU)_{\QQ}\lra \wh\Pic(\CU)_{\QQ}.$$

An adelic $\QQ$-line bundle over $\CU$ is called \emph{strongly nef} if it is isomorphic to a Cauchy sequence
$(\CL, (\CX_i,\overline \CL_i, \ell_{i})_{i\geq 1})$ such that $\CLL_i$ is nef over $\CX_i$ for all $i\geq1$.
An adelic $\QQ$-line bundle $\CLL$ over $\CU$ is called \emph{nef} if there exists a strongly nef adelic $\QQ$-line bundle $\CMM$ over $\CU$ such that $\CLL+\epsilon\CMM$ is strongly nef for all positive rational numbers $\epsilon$.
An adelic $\QQ$-line bundle over $\CU$ is called \emph{integrable} if it is isometric to 
$\CLL_1- \CLL_2$ for two {strongly nef} adelic $\QQ$-line bundle $\CLL_1$ and $\CLL_2$ over $\CU$.

\subsubsection*{Adelic line bundles over finitely generated fields}

Let $K$ be a finitely generated field over $\QQ$ of transcendental degree $d$.
Let $X$ be a quasi-projective variety over $K$ of dimension $n$.
By an \emph{quasi-projective arithmetic model} of $X$, we mean a morphism $i:X\to \CU$ to a quasi-projective arithmetic variety $\CU$ which is injective between the underlying spaces and induces isomorphisms between the local rings.

Define
$$\wh \Picc (X)_{\QQ}:=\lim _{\substack{\lra \\ \CU}}\wh \Picc (\CU)_{\QQ},$$
$$\wh \Pic (X)_{\QQ}:=\lim _{\substack{\lra \\ \CU}}\wh \Pic (\CU)_{\QQ}.$$
An object of $\wh \Picc (X)_{\QQ}$ is called an \emph{adelic $\QQ$-line bundle over $X$}.

There are canonical forgetful maps 
$$
\wh\Pic(X)_{\QQ} \lra \Pic(X)_{\QQ}, \quad\ 
\wh\Picc(X)_{\QQ} \lra \Picc(X)_{\QQ}.
$$ 
These are induced by the forgetful map
$$
\wh\Picc(\CU)_{\QQ} \lra \Picc(\CU)_{\QQ},\quad 
(\CL, (\CX_i,\overline \CL_i, \ell_{i})_{i\geq 1})\longmapsto \CL.
$$ 
As a convention, we usually write an object of $\wh\Picc(X)_{\QQ}$ in the form $\overline L$, where $L$ is understood to be the image of $\overline L$ in $\Picc(X)_\QQ$.
We often refer $L$ as the \emph{underlying $\QQ$-line bundle of} $\overline L$, and refer $\overline L$ as an \emph{arithmetic extension of} $L$.

\subsubsection*{Simplified notations}

As mentioned at the end of \cite[\S2.5]{YZ2}, the groups $\wh\Div(\cdot)_{\QQ}$, 
$\wh\Pic(\cdot)_{\QQ}$ and the category $\wh\Picc(\cdot)_{\QQ}$ are the base change to $\QQ$ of their integral versions. 

As this paper only concerns with $\wh \Pic (\cdot)_{\QQ}$, so we will abbreviate 
$$
\wh \Div (\CU)=\wh \Div (\CU)_{\QQ}, \qquad
\wh \Pic (\CU)=\wh \Pic (\CU)_{\QQ}, \qquad
 \wh \Picc (\CU)=\wh \Picc (\CU)_{\QQ},
$$
and 
$$
\wh \Div (X)=\wh \Div (X)_{\QQ}, \qquad
\wh \Pic (X)=\wh \Pic (X)_{\QQ}, \qquad \wh \Picc (X)=\wh \Picc (X)_{\QQ}.
$$
For convenience, we may call an (adelic) $\QQ$-line bundle simply an (adelic) line bundle.

By the direct limit, we have notions of \emph{nef} adelic $\QQ$-line bundles and \emph{integrable} adelic $\QQ$-line bundles over $X$. 
These gives a sub-semigroup $\wh \Pic (X)_\nef$ and a subgroup $\wh \Pic(X)_\intb$ of 
$\wh \Pic (X)$.

The definitions also work for $X=\Spec K$. 
We will simply write $\wh \Pic (\Spec K)$ as $\wh \Pic (K)$. 
Write $\wh \Pic (\Spec K)_\nef$ and $\wh \Pic(\Spec K)_\intb$ similarly.

\subsubsection*{Intersection theory}
By \cite[\S4.1]{YZ2},
there is an (absolute) intersection product
$$\wh\Pic (K)_\intb^{d+1}\lra \BR, \qquad  (\overline H_1, \cdots, \overline H_{d+1})\longmapsto\overline H_1 \cdots \overline H_{d+1},$$
and a relative intersection product
$$\wh \Pic (X)_\intb^{n+1}\lra \wh\Pic (K)_\intb, \qquad 
(\overline L_1, \cdots, \overline L_{n+1})\longmapsto \pi_*( \overline L_1 \cdots \overline L_{n+1}).
$$
The relative intersection product is defined as the limit of the Deligne pairing.

\subsection{The inequality}

We first deduce the inequality of Theorem \ref{hodge} from that of \cite[Thm. 3.2]{YZ1}. 
The goal is to prove 
$$\pi_*(\overline M^2\cdot \overline L_1\cdots \overline L_{n-1}) \preceq 0$$
in $\wh \Pic(K)_\intb$ under the assumption
$$M\cdot L_1\cdots L_{n-1}=0$$
over $X$. 
Here $\overline M$ is integrable over $X$, $\overline L_1, \cdots, \overline L_{n-1}$ are nef over $X$, and the underlying line bundles
$ L_1, \cdots,  L_{n-1}$ are big over $X$. 

Our first step is to reduce it to the case that  $\overline L_1, \cdots, \overline L_{n-1}$ are strongly nef over $X$.
This follows from the idea at the beginning of \cite[\S 3.3]{YZ1}. 
In fact, let $\overline A$ be a nef and big line bundle over $X$ such that $\overline L_i'=\overline L_i+\epsilon \overline A$ is strongly nef for all rational numbers $\epsilon>0$ and for all $i=1,\cdots, n-1$.
Set $\overline M'=\overline M+\delta \overline A$ with underlying line bundle $M'=M+\delta A$.
Here $\delta=\delta(\epsilon)$ is a number such that 
$$
M'\cdot L_1'\cdots L_{n-1}'=(M+\delta A)\cdot L_{1}'\cdots L_{n-1}'=0.
$$
It determines
$$\delta= -\frac{M\cdot L_{1}'\cdots L_{n-1}'}{A\cdot L_{1}'\cdots L_{n-1}'}.$$
As $\epsilon\to 0$, we have $\delta\to 0$ since
$$
M\cdot L_{1}'\cdots L_{n-1}'\lra 
M\cdot L_{1}\cdots L_{n-1} =0
$$
and 
$$
A\cdot L_{1}'\cdots L_{n-1}'\lra
A\cdot L_{1}\cdots L_{n-1}\\
>0.
$$
The last inequality uses the assumption that $L_1,\cdots, L_{n-1}, A$ are 
big and nef over $X$.

Therefore, the inequality for $(\overline M, \overline L_1,\cdots, \overline L_{n-1})$ is implied by that for $(\overline M', \overline L_1',\cdots, \overline L_{n-1}')$.
Thus we can assume that  $\overline L_1, \cdots, \overline L_{n-1}$ are strongly nef over $X$ in the following.

As our second step, by approximation, it suffices to prove the following assertion.

\emph{
Let $\pi:\CX\to \CB$ be a projective and flat morphism of projective arithmetic varieties. 
Write $\dim \CB=d+1$ and $\dim \CX=n+d+1$. 
Let $\CMM$ be a hermitian line bundle on $\CX$,
$(\CLL_1,\cdots, \CLL_{n-1})$ be nef hermitian line bundles on $\CX$ with big generic fibers on $X$,
and $(\CHH_1,\cdots, \CHH_{d})$ be nef hermitian line bundle on $\CB$.   
Assume that the generic fiber $\CL_{i,\eta}$ is big on the generic fiber $\CX_\eta$ of $\CX$ above the generic point $\eta$ of $\CB$ for every $i=1, \cdots, n-1$.
If
$$\CM_\eta\cdot \CL_{1,\eta}\cdots \CL_{n-1,\eta}=0,$$
then 
$$
\CMM^2\cdot \CLL_1\cdots \CLL_{n-1}\cdot
\pi^* \CHH_1 \cdots\pi^* \CHH_{d}\leq 0.
$$
}

Our third step is to apply the above $(\epsilon,\delta)$-trick again to get stronger positivity. 
In fact, we can assume that each $\CHH_j$ is ample  on $\CB$ since nef line bundles are limits of ample line bundles. 
For simplicity, denote $\CLL_{n-1+j}=\pi^*\CHH_j$ for $j=1,\cdots, d$. 
Fix an ample hermitian line bundle $\CAA$ on $\CX$. 
Take a small rational number $\epsilon>0$. 
Set $\CMM'=\CMM+\delta \CAA$
and $\CLL_i'= \CLL_i+\epsilon \CAA$ for $i=1,\cdots d+n-1$. 
Here $\delta=\delta(\epsilon)$ is a number such that 
$$
\CM_\QQ'\cdot \CL_{1,\QQ}'\cdots \CL_{d+n-1,\QQ}'=(\CM_\QQ+\delta \CA_\QQ)\cdot \CL_{1,\QQ}'\cdots \CL_{d+n-1,\QQ}'=0.
$$
It determines
$$\delta= -\frac{\CM_\QQ\cdot \CL_{1,\QQ}'\cdots \CL_{d+n-1,\QQ}'}{\CA_\QQ\cdot \CL_{1,\QQ}'\cdots \CL_{d+n-1,\QQ}'}.$$
As $\epsilon\to 0$, we have $\delta\to 0$ since
\begin{multline*}
\CM_\QQ\cdot \CL_{1,\QQ}'\cdots \CL_{d+n-1,\QQ}'\lra 
\CM_\QQ\cdot \CL_{1,\QQ}\cdots \CL_{d+n-1,\QQ} \\
=(\CM_\eta\cdot \CL_{1,\eta}\cdots \CL_{n-1,\eta})(\CH_{1,\QQ}\cdots \CH_{d,\QQ})
=0
\end{multline*}
and 
\begin{multline*}
\CA_\QQ\cdot \CL_{1,\QQ}'\cdots \CL_{d+n-1,\QQ}'\lra
\CA_\QQ\cdot \CL_{1,\QQ}\cdots \CL_{d+n-1,\QQ} \\
=(\CA_\eta\cdot \CL_{1,\eta}\cdots \CL_{n-1,\eta})(\CH_{1,\QQ}\cdots \CH_{d,\QQ})
>0.
\end{multline*}
The last inequality uses the assumption that $\CL_{i,\eta}$ is 
big and nef for each $i$.

Applying \cite[Thm. 3.2]{YZ1} to the arithmetic variety $\CX$ over $\ZZ$, 
we have
$$\CMM'^2\cdot\CLL_{1}'\cdots \CLL_{d+n-1}'\leq 0.$$
Set $\epsilon\to 0$. We have
$$\CMM^2\cdot\CLL_{1}\cdots \CLL_{d+n-1}\leq 0.$$
It proves the result.

\subsection{Equality: vertical case}

An adelic line bundle $\overline L\in \wh\Pic(X)_\intb$ is called \emph{vertical} if the underlying line bundle $L$ is isomorphic to the trivial line bundle $\CO_X$. Denote by $\wh\Pic(X)_\vert$ the group of vertical adelic line bundles on $X$.

Now we prove the equality part of the theorem in the vertical case.
Recall that:
\begin{itemize}
\item $K$ is a finitely generated field over $\QQ$ of transcendental degree $d\geq 0$;
\item $X$ is a normal projective variety of dimension $n\geq 1$ over $K$;
\item $\overline M\in \wh\Pic(X)_\vert$ is vertical;
\item $\overline L_1, \cdots, \overline L_{n-1}\in \wh\Pic(X)_\intb$
with $\overline L_i\gg 0$;
\item  $\overline M$ is $\overline L_i$-bounded for each $i$;
\item The equality 
$$\pi_*(\overline M^2\cdot \overline L_1\cdots \overline L_{n-1})\equiv 0$$
holds in $\wh \Pic(K)_\intb$.
\end{itemize}
We need to prove $\overline M\in \pi^*\wh\Pic(K)_\intb$.

By \cite[Lem. 2.3.3]{YZ2}, there is a quasi-projective arithmetic model $\CU\to \CV$ of $X\to \Spec K$, i.e., a projective and flat morphism of quasi-projective arithmetic varieties with generic fiber $X\to \Spec K$, such that 
$$\overline M, \overline L_1, \cdots, \overline L_{n-1}\in \wh\Pic(\CU)_\intb.$$
For any horizontal closed integral subscheme $\CW$ of $\CV$ of dimension $e+1$, 
we get a projective and flat morphism $\CU_\CW\to \CW$, and it defines the groups $\wh\Pic(\CU_\CW)_\intb$ and $\wh\Pic(\CW)_\intb$. 
There are natural pull-back maps 
$$
\wh\Pic(\CU)_\intb \to \wh\Pic(\CU_\CW)_\intb, \quad
\wh\Pic(\CV)_\intb \to \wh\Pic(\CW)_\intb.
$$
We first prove the following result. 
\begin{lem}
For any $\overline H_1, \cdots, \overline H_e\in \wh\Pic(\CV)_\intb$, one has
$$
(\overline M|_{\CU_\CW})^2\cdot (\overline L_1|_{\CU_\CW})\cdots (\overline L_{n-1}|_{\CU_\CW})\cdot 
(\overline H_1|_{\CW})\cdots (\overline H_{e}|_{\CW} )=0. 
$$
\end{lem}
\begin{proof}
By induction, we can assume that $\CW$ has codimension one in $\CV$.
We need to prove
$$
(\overline M|_{\CU_\CW})^2\cdot (\overline L_1|_{\CU_\CW})\cdots (\overline L_{n-1}|_{\CU_\CW})\cdot 
(\overline H_1|_{\CW})\cdots (\overline H_{d-1}|_{\CW} )=0. 
$$
By approximation, we can assume that there is a projective model $\CX\to \CB$
 of $\CU\to \CV$ such that $\overline H_i\in \wh\Pic(\CB)$ for every $i=1,\cdots, d-1$.
Denote by $\mathcal C$ the Zariski closure of $\CW$ in $\CB$. 
Then $\CX_{\mathcal C}\to \mathcal C$ is a projective model of $\CU_\CW\to \CW$. 

By assumption, for any $\overline H_d\in \wh\Pic(\CB)$, we have $\CI\cdot \overline H_d=0.$
Here we denote
$$\CI:= \overline M^2\cdot \overline L_1\cdots \overline L_{n-1}\cdot \overline H_1\cdots \overline H_{d-1},$$
which is a formal notation to ease the burden of the notations and has concrete meanings when intersecting it with other line bundles (or 1-cycles).
Then the intersection of $\CI$ with any vertical class of $\CB$ is zero. 

Now assume that the finite part $\CH_d$ of $\overline H_d$ is ample on $\CB_\QQ$. 
After replacing $\overline H_d$ by a multiple if necessary, we can assume that there is a section $s$ of $\CH_d$ vanishing on $\mathcal C$.
It follows that we can write 
$$
\div(s)=\sum_{i=0}^r a_i \mathcal C_i, \quad a_i\geq 0. 
$$ 
Here $\mathcal C_0=\mathcal C$ and $a_0>0$. 
By definition of intersection numbers, 
$$
\CI\cdot \overline H_d= \sum_{i=0}^r a_i \CI\cdot \mathcal C_i
-\int_{\CB(\CC)} \log\|s\| \omega_{\CI}. 
$$
Here the integral is a formal intersection of $\log\|s\|$ with $\CI$, which is zero since $\CI$ has zero intersection with any vertical class. 
Furthermore, $ \CI\cdot \mathcal C_i=0$ if $\mathcal C_i$ is vertical, 
and $ \CI\cdot \mathcal C_i\leq0$ by the inequality part of Theorem \ref{hodge}. 
Hence, $\CI\cdot \overline H_d=0$ forces
$\CI\cdot \mathcal C=0.$
It is exactly the equality that we need to prove. 
\end{proof}

Set $\dim \CW=1$ in the lemma. 
Then the function field of $\CW$ is a number field. 
Apply the main theorem of \cite{YZ1}, we conclude that 
$$\overline M|_{\CU_\CW}\in \pi^*\wh\Pic(\CW)_\intb.$$
To imply $\overline M\in \pi^*\wh\Pic(K)_\intb$, 
we first re-interpret it in terms of Berkovich spaces.

By \cite[Proposition 3.5.1]{YZ2}, we have canonical injections 
$$\wh\Pic(X)\hookrightarrow \wh\Pic(X^\ran)_{\QQ}, 
\quad \wh\Pic(K)\hookrightarrow \wh\Pic((\Spec K)^\ran)_{\QQ}.$$
Here $X^\ran$ is the disjoint union, over all places $v$ of $\QQ$, of the Berkovich space $X_v^\an$ associated to the scheme 
$X_{\QQ_v}=X\times_{\QQ}\QQ_v$ over the complete field $\QQ_v$.

We claim that  $\overline M\in \pi^*\wh\Pic(K)_\intb$ is equivalent to  
$\overline M\in \pi^*\wh\Pic((\Spec K)^\ran)_{\QQ}$. 
In fact, assume the later. 
If there is a rational point $s\in X(K)$, then we would have 
$\overline M= \pi^* \overline M_0$ where 
$\overline M_0=s^* \overline M$ lies in $\wh\Pic(K)_\intb$. 
The identity can be checked in $\wh\Pic(X^\ran)$. 
In general, taking any $x\in X(\overline K)$, we have 
$\overline M=\pi^* \fh_{\overline M}(x)$ with $\fh_{\overline M}(x)\in\wh\Pic(K)_\intb$.

Hence, it suffices to prove 
$\overline M\in \pi^*\wh\Pic((\Spec K)^\ran)_{\QQ}$. 
Fixing an isomorphism from $M$ to the trivial sheaf $\CO_X$, the metric of $\overline M$ corresponds to a continuous function 
$$
-\log\|1\|_{\overline M}: X^\ran \lra \RR.
$$
It suffices to prove that $\log\|1\|_{\overline M}$ is constant on the fiber of any point of 
$(\Spec K)^\ran$. 

Let $\CU\to \CV$ and $\CW$ be as above. 
Then $\log\|1\|_{\overline M}$ extends to $\CU^\ran$. 
By $\overline M|_{\CU_\CW}\in \pi^*\wh\Pic(\CW)_\intb$, we see that 
$\log\|1\|_{\overline M}$ is constant on the fibers of $\CU^\ran\to \CV^\ran$ above 
$w_v$ for any 
closed point $w$ of $\CV_\QQ$ and any place $v$  of $\QQ$.
Here $w_v$ denotes the finite subset of classical points of $\CV_{v}^\ran$ corresponding to the finite subset of closed points $\CV_{\QQ_v}$ determined by $w$. 
By the density of $\{w_v\}_w$ in $ \CV_{\QQ_v}^\ran$, we conclude that 
$\log\|1\|_{\overline M}$ is constant on any fiber of $\CU^\ran\to \CV^\ran$. 
Then it is constant on any fiber of $X^\ran\to (\Spec K)^\ran$. 
It finishes the proof.

\subsection{Equality: case of curves}

For a curve over a finitely generated field, the most elegant Hodge index theorem should be a direct extension of the theorem of Faltings \cite{Fal} and Hriljac \cite{Hr}, which gives an identity between arithmetic intersections and Neron--Tate heights. This task is finished in \cite[Thm. 6.5.1]{YZ2}. 
In this section, we will use this theorem to prove Theorem \ref{hodge} for curves. 

The following is the statement of \cite[Thm. 6.5.1]{YZ2}, combined with the Northcott theorem of \cite[Thm. 5.3.1]{YZ2}.

\begin{thm} \label{hodge2}
Let $K$ be a finitely generated field over $\BQ$, and let $\pi: X\to \Spec K$ be a smooth,  projective, and geometrically connected curve.
Let $M\in \Pic(X)_\QQ$ with $\deg M=0$. 

Then there is an adelic line bundle $\overline M_0\in \wh\Pic(X)_\intb$ with underlying line bundle $M$ such that $\pi_*(\overline M_0\cdot\overline V)\equiv 0$ for any $\overline V\in \wh\Pic(X)_\vert$. 

Moreover, $\pi_*(\overline M_0\cdot\overline M_0)\preceq 0$ in 
$\wh\Pic(K)_\intb$; the equality $\pi_*(\overline M_0\cdot\overline M_0)\equiv 0$ holds if and only if $M=0$ in $\Pic(X)_\QQ$.
\end{thm}

It is easy to see that the theorem implies Theorem \ref{hodge} for curves. 
In fact, define $\overline N\in \wh\Pic(X)_\vert$ by 
$$
\overline M=\overline M_0+ \overline N. 
$$ 
Note that $\pi_*(\overline M_0\cdot\overline N)\equiv 0$. 
We have 
$$
\pi_*(\overline M\cdot\overline M)
\equiv \pi_*(\overline M_0\cdot\overline M_0)
+\pi_*(\overline N\cdot\overline N)
\preceq 0. 
$$
Here $\pi_*(\overline M_0\cdot\overline M_0)\preceq 0$ by Proposition \ref{hodge2} and  
$\pi_*(\overline N\cdot\overline N)\preceq 0$ by the vertical case of Theorem \ref{hodge}. 

If the equality holds, then $M=0$ by Theorem \ref{hodge2}. By the vertical case of Theorem \ref{hodge}, we conclude that $\overline M\in  \pi^*\wh \Pic(K)_\intb$.

\subsection{Equality: general case}

The proof of the equality part of Theorem \ref{hodge} is almost identical to that in \cite{YZ1}. We have already treated the case $n=1$, so we assume $n\geq 2$ in the following.  

\subsubsection*{Argument on the generic fiber}

Assume the conditions in the equality part of Theorem \ref{hodge}, which particularly includes
$$\pi_*(\overline M^2\cdot \overline L_1\cdots \overline L_{n-1}) \equiv 0.$$
We first show that $M$ is numerically trivial on $X$ by the condition  
$\overline L_{n-1}\gg 0$. 

The condition asserts that $\overline L_{n-1} '=\overline L_{n-1} - \overline N$ is nef for some $\overline N\in \wh\Pic(\QQ)$ with $\wh\deg(\overline N)>0$.
Then 
$$\pi_*(\overline M^2\cdot \overline L_1 \cdots \overline L_{n-2}\cdot \overline L_{n-1})
=\pi_*(\overline M^2\cdot \overline L_1 \cdots \overline L_{n-2}\cdot \overline L_{n-1}')
+(M^2\cdot  L_1 \cdots  L_{n-2})\,\overline N.$$
Applying the inequality of the theorem to $(\overline M, \overline L_1, \cdots, \overline L_{n-2}, \overline L_{n-1}')$, we have
$$\pi_*(\overline M^2\cdot \overline L_1 \cdots \overline L_{n-2}\cdot \overline L_{n-1}') \preceq 0.$$
By the Hodge index theorem on $X$ in the geometric case, we have
$$M^2\cdot  L_1 \cdots  L_{n-2}\leq 0.$$
Hence,
$$\pi_*(\overline M^2\cdot \overline L_1 \cdots \overline L_{n-2}\cdot \overline L_{n-1}')\equiv 0, \quad  M^2\cdot  L_1 \cdots  L_{n-2}
=0.$$

On the variety $X$, we have
$$
M\cdot  L_1 \cdots  L_{n-2}\cdot L_{n-1}=0,
\quad M^2\cdot  L_1 \cdots  L_{n-2}=0.
$$
By the Hodge index theorem on normal algebraic varieties, we conclude that $M$ is numerically trivial. 
See \cite[Appendix 1]{YZ1}.

\subsubsection*{Numerically trivial case}

We have proved that $M$ is numerically trivial on $X$, and now we continue to prove that $M$ is a torsion line bundle. 
Then a multiple of $\overline M$ is vertical and has already been treated. 
As in \cite{YZ1}, the key is still the variational method. 

\begin{lem} 
Let $\overline M, \overline L_1, \cdots, \overline L_{n-1}$ be integrable adelic line bundles on $X$ such that the following conditions hold:
\begin{enumerate}[(1)]
\item $M$ is numerically trivial on $X$;
\item $\overline M$ is $\overline L_i$-bounded for every $i$;
\item $\pi_*(\overline M^2\cdot \overline L_1\cdots \overline L_{n-1})\equiv 0$.
\end{enumerate}
Then for any nef adelic line bundles $\overline L_i^0$ on $X$ with underlying bundle $L_i^0$ numerically equivalent to $L_i$, and any integrable adelic line bundle $\overline M'$ with numerically trivial underlying line bundle $M'$, the following are true:
$$\pi_*( \overline M\cdot \overline M'\cdot \overline L_1^0\cdots \overline L_{n-1}^0) \equiv 0,$$
$$\pi_*(\overline M^2\cdot \overline M'\cdot \overline L_1^0\cdots \overline L_{n-2}^0) \equiv 0.$$
\end{lem}
\begin{proof}
The proof is similar to its counterpart in \cite{YZ1}. 
For example for the first equality, it suffices to prove  
$$
\overline M\cdot \overline M'\cdot \overline L_1^0\cdots \overline L_{n-1}^0 \cdot 
\pi^*\overline H_1\cdots\pi^*\overline H_d=0
$$
for any nef $\overline H_1, \cdots, \overline H_d\in \wh\Pic(\Spec K)_\intb$. 
For fixed $\overline H_1, \cdots, \overline H_d$, the intersection numbers still satisfy the Cauchy--Schwartz inequality. The proof can be carried here.
\end{proof}

Go back to the equality part of Theorem \ref{hodge}.  
Apply Bertini's theorem. Replacing $\overline L_{n-1}$ by a positive multiple if necessary, there
is a section $s\in H^0(X,L_{n-1})$ such that $Y=\div(s)$ is an integral subvariety of $X$, regular in codimension one. 
Then we have 
$$
\pi_*(\overline M^2\cdot \overline L_1\cdots \overline L_{n-2}\cdot \overline L_{n-1})
\equiv
\pi_*(\overline M^2\cdot \overline L_1\cdots \overline L_{n-2}\cdot Y).
$$ 
In fact, the difference of two sides is the limit of the intersection of $\overline M^2\cdot \overline L_1\cdots \overline L_{n-2}$ with vertical classes, so it vanishes by the second equality of the lemma. 
Hence, 
$$
\pi_*(\overline M^2\cdot \overline L_1\cdots \overline L_{n-2}\cdot Y)\equiv 0.
$$ 
By the Lefschetz hyperplane theorem, we can assume that 
$\Pic^0(X)_\QQ\to \Pic^0(Y)_\QQ$ is injective. 
It reduces the problem to $Y$.
The proof is complete since we have already treated the case of curves.

\section{Preperiodic points of algebraic dynamics}

The goal of this section is to prove Theorem \ref{dynamicsmain} on the rigidity of preperiodic points. The idea is very similar to the number field case in \cite{YZ1}.

\subsection{Review on admissible line bundles}

Let us first recall the theory of admissible adelic line bundles for polarizable algebraic dynamical systems over finitely generated fields in \cite[\S6.4]{YZ2}, which generalizes the results of \cite{Zh2, YZ1} over number fields. 

Let $K$ be a finitely generated field over $\BQ$.
Let $(X, f, L)$ be a \emph{polarized dynamical system} over a $K$, i.e.,
\begin{itemize}
\item $X$ is a projective variety over $K$;
\item $f:X\to X$ is a morphism over $K$;
\item $L\in \Pic(X)_\QQ$ is an ample line bundle such that
$f^*L=qL$ from some $q>1$.
\end{itemize}
We will further assume that $X$ is normal, which can always achieved by taking the normalization of $X$.

By \cite[Thm. 6.1.1]{YZ2}, there is an adelic $\QQ$-line bundle 
$\overline L_f\in \wh\Pic(X)_{\nef}$ extending $L$ and with $f^*\overline L_f=q\overline L_f$.  
This still follows from Tate's limiting argument.
By \cite[Thm. 6.1.1(2)]{YZ2}, $\overline L_f$ is actually strongly nef over $X$. 

For any closed $\overline K$-subvariety $Z$ of $X$, we have the canonical height
$$
\fh_{f}(Z)
=\fh_{\overline L_f}(Z):
= \frac{ \left\langle \overline L_f|_{\widetilde Z}\right\rangle^{\dim Z+1} }
{(\dim Z+1)\deg_{L}(\widetilde Z) }\ \in \wh \Pic(K)_{\nef}.$$
Here $\wt Z$ is the image of $Z$ in $X$. 
In particular, we have a height function
$$
\fh_f: X(\overline K) \lra \wh \Pic(K)_{\nef}.
$$
These heights can also be obtained by Tate's limiting argument. 

The height function $\fh_f$ is $f$-invariant. 
Moreover, for any point $x\in X(\overline K)$, $\fh_f(x)=0$ if and only if $x$ is preperiodic under $f$.

Fix $\overline H_1,\cdots, \overline H_d\in \wh\Pic(K)_{\nef}$, where $d$ is the transcendental degree of $K$ over $\QQ$. 
For any closed $\overline K$-subvariety $Z$ of $X$,  {the Moriwaki height of $Z$ with respect to $\overline L$ and 
$(\overline H_1,\cdots, \overline H_d)$} is
$$h_{\overline L_f}^{\overline H_1,\cdots, \overline H_d}(Z)
:=\fh_{\overline L_f}(Z)\cdot \overline H_1\cdots \overline H_d
=\frac{  \overline L_f|_{\widetilde Z}^{\dim Z+1} \cdot \overline H_1\cdots \overline H_d}
{(\dim Z+1)\deg_{L}(\widetilde Z) }.$$
Here the intersection numbers are taken in $\wh\Pic(\widetilde Z)_{\intb}$.
It gives a real-valued height function.

In \cite[Thm. 6.4.2]{YZ2}, the notion of $f$-invariant 
adelic line bundles is extended to $f$-admissible adelic line bundles.
Namely, the projection 
$$
\wh\Pic(X)\lra \Pic(X)_\QQ
$$
has a unique section 
$$
M\longmapsto \overline M_f
$$
as $f^*$-modules. 

Moreover, if $M\in \Pic_f(X)_\BQ$ is $f$-pure of weight 2 and ample, then $\overline M_f$ is nef. 
This generalizes the canonical height functions $\fh_{\overline L_f}$ and
$h_{\overline L_f}^{\overline H_1,\cdots, \overline H_d}$
 to 
$\fh_{\overline M_f}$ and
$h_{\overline M_f}^{\overline H_1,\cdots, \overline H_d}$.

Note that the projection 
$$
\Pic(X)_\QQ\lra \NS(X)_\QQ
$$
also has a unique $f^*$-equivariant section 
$$
\ell_f:  \NS(X)_\QQ \lra \Pic(X)_\QQ.
$$
Denote by
$$
\wh\ell_f:  \NS(X)_\QQ \lra \wh\Pic(X).
$$
the composition of the two sections.

\subsection{Preperiodic points}
The goal of this section is to prove  Theorem \ref{dynamicsmain}. 
By Lefschetz principle, we can assume that $K$ is finitely generated over $\QQ$.  
The following result refines the theorem. 
The condition of $X$ being normal can be obtained by taking a normalization. 

\begin{thm} \label{dynamicsrefine}
Let $X$ be a normal projective variety over a finitely generated field $K$. For any $f,g\in \DS(X)$, the following are equivalent:
\begin{itemize}
\item[(1)] $\Prep(f)= \Prep(g)$;
\item[(2)] $g\Prep(f)\subset\Prep(f)$;
\item[(3)] $\Prep(f)\cap \Prep(g)$ is Zariski dense in $X$;
\item[(4)] $\wh\ell_f=\wh\ell_g$ as maps from $\NS(X)_\QQ$ to 
$\wh\Pic(X)$.
\end{itemize}
\end{thm}

As in the number field case, we prove (1) $\Rightarrow$ (2) $\Rightarrow$ (3) $\Rightarrow$ (4) $\Rightarrow$ (1). 
The proofs of the easy directions are similar to the number field case.
In the implication (2) $\Rightarrow$ (3), we need the finiteness of 
$$\Prep (f,r):=\{x\in \Prep(f)\ |\ \deg(x)<r \}.$$
It is given by Northcott's property of the Moriwaki height (cf. \cite[Thm. 5.3.1]{YZ2}). 
In the following, we prove the hard implication (3) $\Rightarrow$ (4).

\subsubsection*{Applying the Hodge index theorem}

Assume that $\Prep (f)\cap\Prep (g)$ is Zariski dense in $X$. 
As usual, write $n$ for the dimension of $X$ and $d$ for the transcendental degree of $K$ over $\QQ$. 
We need to prove $\wh\ell_f(\xi)=\wh\ell_g(\xi)$ for any $\xi\in\NS(X)_\QQ$. 
By linearity, it suffices to assume that $\xi$ is ample. 

Denote $L=\ell_f(\xi)$ and $M=\ell_g(\xi)$. They are ample $\QQ$-line bundles on $X$. Then the admissible extensions $\overline L_f=\wh\ell_f(\xi)$ and $\overline M_g=\wh\ell_g(\xi)$ are nef by \cite[Thm. 6.4.2]{YZ2}. 

Consider the sum $\overline N=\overline L_f+\overline M_g$, which is still nef.
By the fundamental inequality in \cite[Thm. 5.3.2]{YZ2}, 
$$\lambda_1^{\overline H}(X,\overline N) \geq h_{\overline N}^{\overline H}(X)
$$
for any  $\overline H\in\wh\Pic (K)_\nef$ satisfying the Moriwaki condition.
Here $\overline H$ is said to satisfy the Moriwaki condition if it is induced by a nef hermitian line bundle on a projective model of $K$, the geometric top self-intersection number $H^{d}>0$, and the arithmetic top self-intersection number $\overline H^{d+1}=0$. 

Note that the essential minimum $\lambda_1^{\overline H}(X,\overline N)=0$
since $h_{\overline N}^{\overline H}$ is zero on
$\Prep (f)\cap\Prep (g)$, which is assumed to be Zariski dense in $X$.
It forces $h_{\overline N}^{\overline H}(X) = 0.$
Write in terms of intersections, we have 
$$
(\overline L_f+\overline M_g)^{n+1}\cdot \overline H^d=0.
$$
Expand it by the binomial formula. 
Note that every term is non-negative. 
It follows that 
$$
\overline L_f^{i}\cdot\overline M_g^{n+1-i}\cdot \overline H^d=0, \quad\forall
i=0,1,\cdots, n+1.
$$
It is true for any $\overline H$ satisfying the Moriwaki condition. 
We can remove the dependence on $\overline H$ by the following result.

\begin{lem} \label{moriwaki condition}
Let $\overline Q \in \wh\Pic (K)_\nef$ be a nef adelic line bundle such that the intersection number 
$\overline Q\cdot \overline H^d=0$ for any 
$\overline H\in\wh\Pic (K)_\nef$ satisfying the Moriwaki condition. 
Then $\overline Q$ is numerically trivial. 
\end{lem}

We will prove the lemma later. 
With the lemma, we have
$$
\pi_*(\overline L_f^{i}\cdot\overline M_g^{n+1-i})\equiv 0, \quad\forall
i=0,1,\cdots, n+1.
$$
Then the proof is similar to the number field case. 
In fact, we have 
$$
\pi_*((\overline L_f-\overline M_g)^2 \cdot (\overline L_f+\overline M_g)^{n-1})
\equiv 0.
$$
We still have
$$(L-M)\cdot (L+M)^{n-1}=0$$
since $L-M\in \Pic^0(X)_\QQ$ is numerically trivial. 
Apply Theorem \ref{hodge} to 
$$(\overline L_f-\overline M_g,\ \overline L_f+\overline M_g).
$$ 
It is trivial that $(\overline L_f-\overline M_g)$ is $(\overline L_f+\overline M_g)$-bounded.
To meet the condition $\overline L_f+\overline M_g\gg 0$,
we can take any $\overline C\in \wh\Pic(\QQ)$ with $\deg(\overline C)>0$, and replace 
$$
(\overline L_f-\overline M_g,\ \overline L_f+\overline M_g)
$$
by 
$$
(\overline L_f-\overline M_g, \ \overline L_f+\overline M_g+\pi^*\overline C).
$$
Then all the conditions are satisfied. 
The theorem implies that 
$$
\overline L_f-\overline M_g \in \pi^*\wh\Pic(K)_\intb.
$$
By evaluating at any point $x$ in $\Prep (f)\cap\Prep (g)$,
we see that 
$$
\overline L_f-\overline M_g =0
$$
in $\wh\Pic(X)_{\intb}$. 
It proves the theorem.

\subsubsection*{Local version}

Now we prove Theorem \ref{dynamicslocal}. 
It can be viewed as a local version of Theorem \ref{dynamicsrefine}. 
For that purpose, we first extend the notion of $f$-admissibility to the local setting. 

Let $\BK$ be either $\CC$ or $\CC_p$. Let $(X, f, L)$ be a {polarizable dynamical system} over $\BK$. Assume that $X$ is normal of dimension $n$.  
The exact sequence
$$0\lra \Pic ^0(X)_\BQ\lra \Pic(X)_\BQ\lra \NS (X)_\BQ\lra 0.$$
still has a natural splitting
$$\ell_f:  \NS (X)_\BQ\lra \Pic(X)_\BQ.$$
In fact, since $\NS (X)$ is a finitely generated $\ZZ$-module, 
we can find a finitely generated subfield $K$ of $\BK$ such that $(X,f)$ and all elements of $\NS(X)$ are defined over $K$. Then the lifting $\ell_f$ is defined, and does not depend on the choice of $K$. We say that elements of $\Pic(X)_\BQ$ in the image of $\ell_f$ are  \emph{$f$-pure of weight $2$}.

Denote by $\wh\Pic(X)$ the group of line bundles $L$ on $X$, with a continuous $\BK$-metric on the corresponding Berkovich space $X^\Ber$. Note that if $\BK=\CC$, it is the usual complex analytic space. 
As in the finitely generated case, we have a unique section 
$$\wh\ell_{f,\BK}:  \NS (X)_\BQ\lra \wh\Pic(X)_\BQ/\RR^\times$$
extending $\ell_f$. The group $\RR^\times$ acts on $\wh\Pic(X)$ by scalar multiplication on the metrics. 

For any $M\in \Pic(X)_\RR$ which is $f$-pure of weight 2,  denote by $\overline M_f$ the image of the algebraic equivalence class of $M$ under $\wh\ell_f$. If $M$ is furthermore ample, then the metric of $\overline M_f$ is semipositive.  
In that case, the equilibrium measure
$$d\mu_f= \frac{1}{\deg(M)} c_1(\overline M_f)^n.$$
In fact, by decomposing $M$ into $f$-eigencomponents. It suffices to check 
$$d\mu_f= \frac{1}{M_1\cdot M_2\cdots M_n} c_1(\overline M_{1,f})\wedge 
c_1(\overline M_{2,f})\wedge\cdots \wedge c_1(\overline M_{n,f})$$
for eigenvectors $M_1,\cdots, M_n$ of $f^*$ in $\Pic(X)_\CC$. 
The identity is understood in terms of linear functionals on the space of complex-value continuous functions on $X^\Ber$. It holds since both sides are $f^*$-invariant, and the uniqueness of $d\mu_f$ coming from Tate's limiting method. 
The following theorem refines Theorem \ref{dynamicslocal}.

\begin{thm} 
Let $\BK$ be either $\CC$ or $\CC_p$ for some prime $p$. 
Let $X$ be a normal projective variety over $\BK$, and let $f,g\in \DS(X)$ be two polarizable algebraic dynamical system over $X$ such that $\Prep(f)\cap \Prep(g)$ is Zariski dense in $X$.
Then $\wh\ell_{f,\BK}=\wh\ell_{g,\BK}$ as maps from $\NS(X)_\QQ$ to 
$\wh\Pic(X)_{\QQ}/\RR^\times$.
\end{thm}

Let us see how to obtain the result from Theorem \ref{dynamicsrefine}. 
Let $K$ be a finitely generated subfield of $\mathbb K$ such that $(X,f,g)$ is defined over $K$; namely, $(X,f,g)$ is isomorphic to the base change of a triple $(X_0,f_0,g_0)$ from $K$ to $\mathbb K$. 
We can further assume that 
 $\NS(X)=\NS(X_0)$, which can be achieved by enlarging $K$ since $\NS(X)$ is finitely generated. 
 
 Consider the inclusion $\eta:K\hookrightarrow \mathbb K$.  
The canonical absolute value on $\mathbb K$ induces a 
point $\eta^\an$ of $(\Spec K)^\an$. 
By definition, the fiber $X^\an_{0,\eta^\an}$ of $X_0^\an$ above $\eta^\an$ is isomorphic to $X^\Ber$. 
For any $\xi\in \NS(X_0)_\QQ$, by Theorem \ref{dynamicsrefine}, we have 
$\wh\ell_{f_0}(\xi)=\wh\ell_{g_0}(\xi)$ in $\wh\Pic(X_0)_\intb$ in the setting of finitely generated fields.
The identity is viewed as an equality of metrics on $X_0^\an$. 
Restricted to the fiber $X^\Ber$, we have 
$\wh\ell_{f,\BK}(\xi)=\wh\ell_{g,\BK}(\xi)$.  
The result is proved.

\subsubsection*{Moriwaki condition}

It remains to prove Lemma \ref{moriwaki condition}. 
It takes a few steps. Assume $\overline Q\in \wh\Pic(\CV)_\intb$ for some quasi-projective arithmetic model $\CV$ of $K$. 

\medskip

\noindent \textbf{Step 1.}
Replacing $\CV$ by an open subscheme if necessary, we can assume that there is a finite morphism $\psi: \CV\to \CV_0$ for some open subschemes $\CV_0$ of $\BP^d_\ZZ$. This follows from Noether's normalization lemma.
The goal of this step is to prove that the height function $h_{\overline Q}$ on $\CV(\overline\QQ)$ associated to $\overline Q$ is identically 0. Namely, for any horizontal closed integral subscheme $\CW$ of dimension one in $\CV$, the restriction $\overline Q|_\CW\in \wh\Pic(\CW)_\intb$ has arithmetic degree 0. 

We first treat the case that $\CW$ has degree 1 over $\ZZ$. By automorphism of 
$\BP^d_\ZZ$, we can further assume that the image of $\CW_\QQ$ is exactly the rational point $W_0=(0,\cdots, 0,1)$ of $\BP^d_\ZZ$. 
Denote by $\CW_0$ the Zariski closure of $W_0$ in $\BP^d_\ZZ$. 
Take the metrized line bundle $\overline \CH_0=(\CO(1), \|\cdot\|_0)$ on $\BP^d_\ZZ$ satisfying the dynamical property that the pull-back of $\overline \CH_0$ by the square map is isometric to $2\overline \CH_0$. 
Note that the Moriwaki condition $\CH_{0,\QQ}^{d}>0$ and $\overline \CH_0^{d+1}=0$ is satisfied.
By the coordinate sections of $\CO(1)$, we see that 
$\overline \CH_0^d$ is represented by the arithmetic 1-cycle $(\CW_0, \mathfrak{g}_0)$ for some positive current $\mathfrak{g}_0$ on $\BP^d(\CC)$. 
Then we have 
$$
0= \overline Q \cdot \psi^* \overline \CH_0^{d}
= \overline Q \cdot \psi^*(\CW_0, \mathfrak{g}_0)
\geq  \overline Q \cdot \psi^*\CW_0
\geq   \overline Q \cdot \CW
\geq 0.  
$$
It follows that $\overline Q \cdot \CW=0$. 
Here we used the nefness of $\overline Q$, and the inequalities can be justified by approximating $\overline Q$ by nef hermitian line bundles on projective models. 

If $\CW$ has higher degree over $\ZZ$, we take a base change from $\ZZ$ to a finite \'etale extension $O_F$ to split $\CW$. Then we can arrange such that $\psi_{O_K}: \CV_{O_K}\to \CV_{0,O_K}$ maps $\CW_\QQ$ to $W_0=(0,\cdots, 0,1)$ of $\BP^d_{O_K}$. The proof still works.

\medskip

\noindent \textbf{Step 2.}
Define $\wt Q$ to be the image of $\overline Q$ under the canonical map $\wh\Pic(\CV)\to \wh\Pic(\CV_\QQ/\QQ)$.
See \cite[\S2.5.5]{YZ2} for the definition of this map, which is obtained applying the base change $\Spec \QQ\to \Spec \ZZ$ to the arithmetic models of $\overline Q$.
The goal of this step is to prove that $\wt Q$ is numerically trivial in 
$\wh\Pic(\CV_\QQ/\QQ)$. In other words, 
$$
\wt Q\cdot   A_1\cdots  A_{d-1}=0
$$
for any $A_1,\cdots,  A_{d-1}\in \Pic(\CB_{m,\QQ})$ and any projective model $\CB_{m}$ of $\CV$.

By definition, we can assume that $\overline Q$ is the limit of a sequence of nef hermitian line bundles $\overline \CQ_m$ on projective models $\CB_m$ of 
$\CV$. We can further assume that $\CB_m$ dominates a projective model $\CB$ of 
$\CV$, and there is an effective arithmetic divisor $\overline \CD=(\CD,g_\CD)$ whose finite part is supported on $\CB\setminus \CV$, such that 
$$
-\epsilon_m \overline \CD 
\leq \overline \CQ_m -\overline Q
\leq \epsilon_m \overline \CD  
$$
from some sequence $\epsilon_m\to 0$.
Here the inequality is understood in terms of effectivity of divisors. 

It follows that the height function associated to 
$\epsilon_m \overline \CD-\overline\CQ_m$ is positive on $\CV(\overline\QQ)$. 
In particular, the height function is bounded below on any complete curve in
$\CB_{m,\QQ}$ which intersects $\CV_{\QQ}$. 
Then the generic fiber $\epsilon_m  \CD_\QQ-\CQ_{m,\QQ}$
is nef on such curves. This implies that $\epsilon_m  \CD_\QQ-\CQ_{m,\QQ}$ is pseudo-effective. 

By Bertini's theorem, it is easy to have
$$
(\epsilon_m  \CD_\QQ-\CQ_{m,\QQ}) \cdot A_1\cdots  A_{d-1}\geq 0
$$
for any ample line bundles $A_1,\cdots, A_{d-1}$ on $\CB_{m,\QQ}$. 
Set $m\to\infty$ and use the nefness of $\CQ_{m,\QQ}$.
We have
$$
\wt Q \cdot A_1\cdots  A_{d-1}= 0.
$$
The result follows by taking linear combinations.

\medskip

\noindent \textbf{Step 3.}
Let $\overline \CH$ be any ample hermitian line bundle on $\CB$. 
Recall the essential minimum
$$
\lambda_1(\overline \CH)
=\lambda_1(\CV_\QQ, \overline \CH)
=\sup_{V'\subset \CV_\QQ}\
\inf_{x\in V'(\overline \QQ)} 
h_{\overline \CH}(x).$$
The supremum runs through all open subschemes $V'$ of $\CV_\QQ$. 
Apply the fundamental inequality in \cite[Thm. 5.3.3]{YZ2} to $\overline \CH+\overline Q$. 
We have  
$$
\lambda_1(\overline \CH+\overline Q)
\geq \frac{1}{(d+1)(\CH_\QQ+\wt Q)^d}(\overline \CH+\overline Q)^{d+1}.
$$ 
By the previous two steps, we end up with
$$
\lambda_1(\overline \CH)
\geq \frac{1}{(d+1)\CH_\QQ^d}(\overline \CH+\overline Q)^{d+1}.
$$ 
Replacing $\overline Q$ by a positive multiple, we see that 
$$
(\overline \CH+t\overline Q)^{d+1}
$$
is bounded for any $t>0$. 
It particularly implies that
$$
\overline Q\cdot \overline \CH^{d}=0.
$$

\medskip

\noindent \textbf{Step 4.} It is formal to show that $\overline Q$ is numerically trivial from the property that $\overline Q\cdot \overline \CH^{d}=0$
for any ample hermitian line bundle $\overline \CH$
on $\CB$.  

In fact, for any ample hermitian line bundles $\CHH_1,\cdots, \CHH_d$ on $\CB$, we have
$$
\overline Q\cdot (t_1\overline \CH_1+\cdots + t_d\overline \CH_d)^{d}=0.
$$
It is true for all positive real numbers $t_1,\cdots, t_d$, which forces 
$$
\overline Q\cdot  \overline\CH_1\cdots \overline \CH_d=0.
$$
By linear combinations, it is true for any hermitian line bundles $\CHH_1,\cdots, \CHH_d$ on $\CB$. By varying $\CB$ and taking limits, it is true for any $\CHH_1,\cdots, \CHH_d$ in $\wh\Pic(K)_\intb$.

\end{document}